\documentclass[a4paper,11pt]{article}
\usepackage{amsmath, amsfonts, amscd, amssymb, amsthm, enumerate, float}

\def\bfC{\mathbf{C}}

\newcommand{\Mat}{\operatorname{M}}
\newcommand{\charac}{\chi}

\newcommand{\id}{\operatorname{id}}

\newcommand{\Ker}{\operatorname{Ker}}
\newcommand{\Root}{\operatorname{Root}}
\newcommand{\Nil}{\operatorname{Nil}}

\newcommand{\End}{\operatorname{End}}
\newcommand{\Hom}{\operatorname{Hom}}

\newcommand{\Vect}{\operatorname{span}}
\newcommand{\im}{\operatorname{Im}}
\newcommand{\res}{\operatorname{res}}

\newcommand{\tr}{\operatorname{tr}}
\newcommand{\Sp}{\operatorname{Sp}}

\def\F{\mathbb{F}}
\def\K{\mathbb{K}}
\def\R{\mathbb{R}}
\def\C{\mathbb{C}}

\def\N{\mathbb{N}}

\renewcommand{\L}{\mathbb{L}}

\def\calA{\mathcal{A}}

\def\calW{\mathcal{W}}

\def\lcro{\mathopen{[\![}}
\def\rcro{\mathclose{]\!]}}

\theoremstyle{definition}
\newtheorem{Def}{Definition}[section]
\newtheorem{Not}[Def]{Notation}

\theoremstyle{plain}
\newtheorem{theo}{Theorem}[section]
\newtheorem{prop}[theo]{Proposition}
\newtheorem{cor}[theo]{Corollary}
\newtheorem{lemma}[theo]{Lemma}

\theoremstyle{plain}
\newtheorem{conj}{Conjecture}

\theoremstyle{remark}

\title{The quadratic sum problem for symplectic pairs}
\author{Cl\'ement de Seguins Pazzis\footnote{Universit\'e de Versailles Saint-Quentin-en-Yvelines, Laboratoire de Math\'ematiques
de Versailles, 45 avenue des Etats-Unis, 78035 Versailles cedex, France}
\footnote{e-mail address: dsp.prof@gmail.com}}

\begin{document}

\thispagestyle{plain}

\maketitle

\begin{abstract}
Let $(b,u)$ be a pair consisting of a symplectic form $b$ on a finite-dimensional vector space $V$ over a field $\F$, and of a $b$-alternating endomorphism $u$ of $V$ (i.e. $b(x,u(x))=0$ for all $x$ in $V$).

Let $p$ and $q$ be arbitrary polynomials of degree $2$ with coefficients in $\F$.
We characterize, in terms of the invariant factors of $u$, the condition that $u$ splits into $u_1+u_2$ for some pair $(u_1,u_2)$ of $b$-alternating endomorphisms such that $p(u_1)=q(u_2)=0$.
\end{abstract}




\section{Introduction}

\subsection{Notation and basic definition}

Throughout the article, we let $\F$ be an arbitrary field  and we choose an algebraic closure $\overline{\F}$ of it.
We denote the characteristic of $\F$ by $\charac(\F)$, and we leave open the possibility that $\charac(\F)=2$.
\emph{All the vector spaces under consideration are assumed finite-dimensional}, and we will never repeat this assumption.

Let $p \in \F[t]$ be a \emph{monic} polynomial of degree $n >0$.
The trace of $p$, denoted by $\tr p$, is defined as the opposite of the coefficient of $p$ on $t^{n-1}$:
it is the sum of all roots of $p$ in $\overline{\F}$, counted with multiplicities.
We also denote by $\Root(p)$ the set of all roots of $p$ in $\overline{\F}$.

Let $\mathcal{A}$ be an $\F$-algebra. An element $x$ of $\mathcal{A}$ is called \textbf{quadratic} whenever
$x^2 \in \Vect_\F(1_\mathcal{A},x)$, i.e.\ $p(x)=0$ for some monic $p \in \F[t]$ of degree $2$.

Finally, we will use the French notation for integers: $\N$ denotes the set of all non-negative integers, and $\N^*$ the set of all positive ones.

\subsection{The problem}

Throughout, we let $p$ and $q$ be fixed monic polynomials of $\F[t]$ with degree $2$, and we set
$$\delta:=\tr p-\tr q.$$
A \textbf{$(p,q)$-difference} in the $\F$-algebra $\mathcal{A}$ is an element $x$ of $\mathcal{A}$ that splits into $x_1-x_2$ where $x_1,x_2$ belong to $\mathcal{A}$ and $p(x_1)=0$ and $q(x_2)=0$. The study of such differences in matrix algebras was started by Hartwig and Putcha
in \cite{HP}, who tackled the case where $p=q=t^2-t$, i.e.\ the differences of pairs of idempotent matrices, over algebraically closed
fields of characteristic not $2$. Later, other cases where studied by Wang and Wu \cite{Wanggeneral,WuWang},
with generalizations to arbitrary fields achieved two decades later by Botha \cite{Bothasquarezero} and the author of the present article \cite{dSPidem2,dSPsumoftwotriang}. It is only recently \cite{dSPregular,dSPsumexceptional} that a complete answer to the $(p,q)$-difference problem in matrix algebras was achieved (with an arbitrary field $\F$ and arbitrary polynomials $p$ and $q$ of degree $2$, possibly irreducible over $\F$).

In a recent article \cite{dSPsquarezeroquadratic}, we started examining similar situations in the context of vector spaces equipped with non-degenerate symmetric or alternating bilinear forms. More precisely, given a vector space $V$ and a non-degenerate symmetric or alternating bilinear form $b$ on $V$, we have characterized all the $b$-selfadjoint (respectively, the $b$-skew-selfadjoint) endomorphisms of $V$ that split into the difference of two square-zero $b$-selfadjoint (respectively, $b$-skew-selfadjoint) endomorphisms of $V$ (the only case left out being the one where $\charac(\F) = 2$).
To be more precise, the characterization is in terms of the fundamental invariants for the class of the endomorphism under the conjugation action of the isometry group of $b$.

We view \cite{dSPsquarezeroquadratic} as the first systematic attempt to tackle the $(p,q)$-difference problems in the context of regular quadratic or symplectic spaces, although
other attempts had been made earlier for specific fields (see e.g.\ \cite{delaCruz} for the $(t^2,t^2)$-difference problem over the complex numbers).

In the present article, we will completely solve the $(p,q)$-difference problem in the context of alternating endomorphisms for a symplectic form $b$. To make things clearer, let $b$ be a symplectic form on a vector space $V$. An endomorphism $u$ of $V$ is called
\textbf{$b$-alternating} whenever the bilinear form $(x,y) \mapsto b(x,u(y))$ is alternating, that is $b(x,u(x))=0$ for all $x \in V$.
In that case $u$ is \textbf{$b$-selfadjoint}, that is
$$\forall (x,y)\in V^2, \; b(u(x),y)=b(x,u(y)).$$
Note that if $\charac(\F) \neq 2$ then every $b$-selfadjoint endomorphism is $b$-alternating, but this fails if $\chi(\F)=2$.

The $b$-alternating endomorphisms that split into $u_1-u_2$, where $u_1$ and $u_2$ are $b$-alternating endomorphisms such that $p(u_1)=q(u_2)=0$,
are called the \textbf{alternating $(p,q)$-differences}. The aim of the present article is to characterize such differences.
To understand what we mean by ``characterize", we recall (see Theorem \ref{theo:Scharlau} on page \pageref{theo:Scharlau}) that the $b$-alternating endomorphisms are classified, up to conjugation by an element of the symplectic group $\Sp(b)$ of $b$, by the invariant factors of the endomorphism $u$.
Since being an alternating $(p,q)$-difference is obviously invariant under conjugating by an element of $\Sp(b)$,
in theory it should be possible to characterize the alternating $(p,q)$-differences in terms of their invariant factors.

Finally, note that the alternating $(p,q)$-differences are the alternating $(p,q(-t))$-\emph{sums}, i.e.\
the endomorphisms that split into $u_1+u_2$ where $u_1$ and $u_2$ are $b$-alternating endomorphisms such that $p(u_1)=q(-u_2)=0$.
It turns out (see \cite{dSPregular}) that the characterization of $(p,q)$-differences in matrix algebras is slightly more elegant than the one of $(p,q)$-sums, so we prefer to frame the problem in terms of the former.

\subsection{The viewpoint of symplectic pairs}

An efficient viewpoint for our problem is the one of \emph{symplectic pairs}.

\begin{Def}
A \textbf{symplectic pair} $(b,u)$ consists of a symplectic bilinear form $b$ on a vector space $V$ and of a $b$-alternating endomorphism $u$ of $V$. Such a pair is called \textbf{trivial} if $V=\{0\}$.

A symplectic pair $(b,u)$ is called a \textbf{symplectic $(p,q)$-difference} when there exist $b$-alternating endomorphisms $u_1$ and $u_2$ such that $p(u_1)=q(u_2)=0$ and $u=u_1-u_2$.
\end{Def}

Two symplectic pairs $(b,u)$ and $(b',u')$, over respective vector spaces $V$ and $V'$, are called \textbf{isometric}, and we write $(b,u) \simeq (b',u')$, whenever there exists a vector space isomorphism $\varphi : V \overset{\simeq}{\rightarrow} V'$, such that
$b'(\varphi(x),\varphi(y))=b(x,y)$ for all $x,y$ in $V$, and $u'(\varphi(x))=\varphi(u(x))$ for all $x \in V$.

The orthogonal direct sum of $(b,u)$ and $(b',u')$ is the symplectic pair $(b \bot b',u \oplus u')$
defined on $V \times V'$. Orthogonal direct sums are compatible with isometry.

As far as our problem is concerned, we have the following easy results:

\begin{prop}
Let $(b,u)$ and $(b',u')$ be isometric symplectic pairs.
Then $(b,u)$ is a symplectic $(p,q)$-difference if and only if $(b',u')$ is a symplectic $(p,q)$-difference.
\end{prop}

\begin{prop}\label{prop:orthogonalsum}
Let $(b,u)$ and $(b',u')$ be symplectic $(p,q)$-differences.
Then $(b,u) \bot (b',u')$ is a symplectic $(p,q)$-difference.
\end{prop}

Let $(b,u)$ be a symplectic pair with underlying vector space $V$.
Let $U$ be a linear subspace of $V$ that is stable under $u$.
Denote by $U^{\bot_b}$ the $b$-orthogonal complement of $U$.
Remember that $b$ induces a symplectic form $\overline{b}$ on the quotient space $U/(U \cap U^{\bot_b})$.
Besides, $U^{\bot_b}$ is stable under $u$ (because $u$ is $b$-selfadjoint) and hence $U \cap U^{\bot_b}$ is also stable under $u$.
Therefore $u$ induces an endomorphism $\overline{u}$ of  $U/(U \cap U^{\bot_b})$, and obviously $\overline{u}$
is $\overline{b}$-alternating. We recover a symplectic pair $(\overline{b},\overline{u})$
and we denote it by $(b,u)^U$: it is called the symplectic pair \textbf{induced} by $(b,u)$ on $U$.
The following basic lemma will be particularly useful:

\begin{lemma}
Let $(b,u)$ be a symplectic pair, with underlying vector space $V$. Assume that $u=u_1-u_2$, where $u_1$ and $u_2$ are $b$-alternating endomorphisms of $V$
such that $p(u_1)=q(u_2)=0$. Let $U$ be a linear subspace of $V$ that is stable under $u_1$ and $u_2$.
Then, $(b,u)^U$ is a symplectic $(p,q)$-difference.
\end{lemma}

\begin{proof}
Note that $u=u_1-u_2$ stabilizes $U$, and hence $(\overline{b},\overline{u}):=(b,u)^U$ is well defined. Next, we see that both $u_1$ and $u_2$ stabilize $U \cap U^{\bot_b}$, and hence they induce $\overline{b}$-alternating endomorphisms $\overline{u_1}$ and $\overline{u_2}$ of
$U/(U \cap U^{\bot_b})$. Noting that $\overline{u}=\overline{u_1}-\overline{u_2}$ and that $p(\overline{u_1})=0=q(\overline{u_2})$, we conclude that
$(b,u)^U$ is a symplectic $(p,q)$-difference.
\end{proof}

Let us now recall the construction of the symplectic extension of an endomorphism.
Let $u$ be an endomorphism of a vector space $V$. Denote by $V^\star:=\Hom(V,\F)$ the dual space of $V$, and
by $u^t : \varphi \in V^\star \longmapsto \varphi \circ u \in V^\star$ the transpose of $u$.

\begin{Def}
On $V \times V^\star$, we have the standard symplectic form
$$S_V : \begin{cases}
(V \times V^\star)^2 & \longrightarrow \F \\
\bigl((x,\varphi),(y,\psi)\bigr) & \longmapsto \varphi(y)-\psi(x)
\end{cases}$$
and the extended endomorphism
$$h(u):=u \oplus u^t.$$
It is easy to check that $h(u)$ is $S_V$-alternating. The pair $(S_V,h(u))$ is called the \textbf{symplectic extension}
of $u$, denoted by $S(u)$.
\end{Def}

One checks (see section 2.2 of \cite{dSPsquarezeroquadratic}) that
$S(v_1 \oplus \cdots \oplus v_n) \simeq S(v_1) \bot \cdots \bot S(v_n)$ for every list $(v_1,\dots,v_n)$
of endomorphisms of (potentially distinct) vector spaces.

\vskip 3mm
We are now able to state the classification of symplectic pairs, as obtained by Scharlau \cite{Scharlaupairs}:

\begin{theo}[Scharlau]\label{theo:Scharlau}
\begin{enumerate}[(a)]
\item Every symplectic pair $(b,u)$ is isometric to a symplectic extension.

\item Given endomorphisms $v_1,v_2$ of vector spaces, the symplectic extensions $S(v_1)$ and $S(v_2)$ are isometric
if and only if $v_1$ and $v_2$ are similar.

\item Two symplectic pairs $(b_1,u_1)$ and $(b_2,u_2)$ are isometric if and only if $u_1$ is similar to $u_2$.
\end{enumerate}
\end{theo}

Point (c) means that the isometry ``class" of $(b,u)$ is controlled by the invariant factors of the endomorphism $u$.
Moreover, since every endomorphism is similar to its transpose, it follows from point (a) in Theorem \ref{theo:Scharlau} that these invariant factors read $p_1,p_1,p_2,p_2,\dots,p_n,p_n,\dots$ for monic polynomials $p_1,p_2,\dots,p_n,\dots$ such that $p_{i+1}$ divides $p_i$ for all $i \geq 1$.
If $(b,u) \simeq S(v)$ for some endomorphism $v$, then $p_1,\dots,p_n,\dots$
are the invariant factors of $v$, and they control the isometry class of $(b,u)$.

The following easy lemma will be very useful to solve our problem:

\begin{lemma}\label{lemma:fromA}
Let $v \in \End(V)$ be a $(p,q)$-difference.
Then $S(v)$ is a symplectic $(p,q)$-difference.
\end{lemma}

\begin{proof}
We write $v=v_1-v_2$ where $v_1$ (resp. $v_2$) is an endomorphism of $V$
that is annihilated by $p$ (resp. by $q$).
One checks that $h(v)=h(v_1)-h(v_2)$ and that $p(h(v_1))=h(p(v_1))=0$ and $q(h(v_2))=h(q(v_2))=0$.
Since $h(v_1)$ and $h(v_2)$ are $S_V$-alternating, this shows that $S(v)$ is a symplectic $(p,q)$-difference.
\end{proof}

In \cite{dSPsquarezeroquadratic}, we proved that, given an endomorphism $v$
of a vector space $V$, the symplectic pair $S(v)$ (which was denoted by $H_{-1,1}(v)$ there)
is a symplectic $(t^2,t^2)$-difference if and only if $v$ is a $(t^2,t^2)$-difference in the algebra
$\End(V)$. Moreover, it was proved that a symplectic pair $(b,u)$ is
a symplectic $(t^2,t^2)$-difference if and only if $u$ is a $(t^2,t^2)$-difference in the algebra
$\End(V)$.

One might conjecture that these two results could be extended to an arbitrary pair
$(p,q)$ of monic polynomials with degree $2$.

\begin{conj}\label{conj:conj1}
For a symplectic pair $(b,u)$ to be a symplectic $(p,q)$-difference, it is necessary and sufficient that
$u$ be a $(p,q)$-difference in $\End(V)$.
\end{conj}

\begin{conj}\label{conj:conj2}
Let $v$ be an endomorphism of a vector space $V$.
For $S(v)$ to be a symplectic $(p,q)$-difference, it is necessary and sufficient that $v$ be a $(p,q)$-difference in $\End(V)$.
\end{conj}

It turns out that both of these conjectures fail. The first one is the easier to dispel.
Assume indeed that $p$ is irreducible and $p=q$, take $V$ as a $2n$-dimensional space over $\F$ for some \emph{odd} integer $n>0$, and
choose a symplectic form $b$ on $V$.
Then $(b,0)$ is a symplectic pair, and $0$ is surely a $(p,q)$-difference in $\End(V)$ because we can choose $u_1 \in \End(V)$ such that $p(u_1)=0$, and then $0=u_1-u_1$.
However, no $b$-alternating endomorphism is annihilated by $p$ because, $p$ being irreducible,
the non-trivial invariant factors of such an endomorphism should all be equal to $p$;
as there must be an even number of such invariant factors, this would lead to $\dim V$ being a multiple of $4$, which is not true.
Thus $(b,0)$ is not a symplectic $(p,p)$-difference.

Conjecture \ref{conj:conj2} is harder to invalidate, and we will wait until Section \ref{section:regular} to give a counterexample.

\vskip 3mm
There are two viewpoints for characterizing symplectic $(p,q)$-differences:
the first one is by invariant factors and the second one is to focus on the \emph{indecomposable}
symplectic differences, which we define as follows.

\begin{Def}
Let $p,q$ be monic polynomials of degree $2$ in $\F[t]$.
A symplectic $(p,q)$-difference $(b,u)$ is called \textbf{indecomposable}
whenever it is non-trivial and not isomorphic to the orthogonal direct sum of two non-trivial
symplectic $(p,q)$-differences.
\end{Def}

In short, a symplectic $(p,q)$-difference $(b,u)$ is indecomposable when the underlying vector space
$V$ is non-trivial and admits no non-trivial splitting $V=V_1 \oplus V_2$ into $b$-orthogonal subspaces, both stable under $u$
and for which the induced pairs $(b,u)^{V_1}$ and
$(b,u)^{V_2}$ are symplectic $(p,q)$-differences.

Obviously, every symplectic $(p,q)$-difference is isomorphic to the orthogonal direct sum of
indecomposable ones. And by Proposition \ref{prop:orthogonalsum}, every symplectic pair that is isomorphic to the orthogonal
direct sum of symplectic $(p,q)$-differences is a symplectic $(p,q)$-difference.
Hence, classifying the symplectic $(p,q)$-differences up to isometry amounts to describing the
indecomposable symplectic $(p,q)$-differences up to isometry.

\subsection{The regular-exceptional dichotomy}

In the study of $(p,q)$-differences, a major role is played by the elements of $\Root(p)-\Root(q)$.

Let $u$ be an endomorphism of a finite-dimensional vector space $V$ over $\F$.
Let us write
$$p(t)=t^2-\lambda t+\alpha \quad \text{and} \quad q(t)=t^2-\mu t+\beta.$$
Remember that $\delta=\tr(p)-\tr(q)$.
The \textbf{fundamental polynomial} of the pair $(p,q)$ is defined as the resultant
$$F_{p,q}(t):=\res_{\F[t]}(p(x),q(x-t)) \in \F[t],$$
which is a polynomial of degree $4$.
More explicitly, if we split $p(z)=(z-x_1)(z-x_2)$ and $q(z)=(z-y_1)(z-y_2)$ in $\overline{\F}[z]$,
then
$$F_{p,q}(t)=\prod_{1 \leq i,j \leq 2}\bigl(t-(x_i-y_j)\bigr)=p(t+y_1)\,p(t+y_2)=q(x_1-t)\,q(x_2-t).$$
Now, let $u$ be an endomorphism of a (finite-dimensional) vector space $V$.
We set
$$E_{p,q}(u):=\underset{n \in \N}{\bigcup} \Ker F_{p,q}(u)^n \quad \text{and} \quad R_{p,q}(u):=\underset{n \in \N}{\bigcap} \im F_{p,q}(u)^n.$$
Hence, $V=E_{p,q}(u) \oplus R_{p,q}(u)$, and the endomorphism $u$ stabilizes both linear subspaces
$E_{p,q}(u)$ and $R_{p,q}(u)$ (this is the Fitting decomposition of $F_{p,q}(u)$).

The endomorphism $u$ is called $(p,q)$-\textbf{exceptional} (respectively, $(p,q)$-\textbf{regular})
whenever $E_{p,q}(u)=V$ (respectively, $R_{p,q}(u)=V$).
In other words, $u$ is $(p,q)$-exceptional (respectively, $(p,q)$-regular)
if and only all the eigenvalues of $u$ in $\overline{\F}$ belong to $\Root(p)-\Root(q)$
(respectively, no eigenvalue of $u$ in $\overline{\F}$ belongs to $\Root(p)-\Root(q)$).

The endomorphism of $E_{p,q}(u)$ (respectively, of $R_{p,q}(u)$)
induced by $u$ is always $(p,q)$-exceptional (respectively, always $(p,q)$-regular)
and we call it the $(p,q)$-exceptional part (respectively, the $(p,q)$-regular part) of $u$.

Since $u$ and $h(u)$ have the same spectrum, we note that $u$ is $(p,q)$-exceptional (respectively, $(p,q)$-regular)
if and only if $h(u)$ is $(p,q)$-exceptional (respectively, $(p,q)$-regular).

Now, if we have a symplectic form $b$ for which $u$ is $b$-alternating, then $u$ is $b$-selfadjoint and it follows that
$E_{p,q}(u)$ and $R_{p,q}(u)$
are $b$-orthogonal, yielding induced symplectic pairs $(b,u)^{E_{p,q}(u)}$ and $(b,u)^{R_{p,q}(u)}$, so that
$$(b,u) \simeq (b,u)^{E_{p,q}(u)} \bot (b,u)^{R_{p,q}(u)}.$$
We say that $(b,u)^{R_{p,q}(u)}$ is the \textbf{$(p,q)$-regular part} of $(b,u)$, and that
 $(b,u)^{E_{p,q}(u)}$ is the \textbf{$(p,q)$-exceptional part} of $(b,u)$.

Before we go on, we need a classical lemma on quadratic elements. Its proof is very simple, so we recall it.

\begin{lemma}[Commutation lemma]\label{lemma:commute}
Let $u_1,u_2$ be elements of an $\F$-algebra $\mathcal{A}$ such that $p(u_1)=0$ and $q(u_2)=0$.
Set $u:=u_1-u_2$. Then $u_1$ and $u_2$ commute with $u^2-\delta u$.
\end{lemma}

\begin{proof}
Write $u_1^\star:=(\tr p)1_\calA-u_1$ and $u_2^\star:=(\tr q)1_\calA-u_2$.
It is natural to set $u^\star:=u_1^\star-u_2^\star=\delta\,1_{\mathcal{A}}-u$ as the pseudo-conjugate of $u$.
The pseudo-norm of $u$ is then
\begin{align*}
-(u^2-\delta u) = uu^\star & =u_1 u_1^\star-u_1 u_2^\star-u_2 u_1^\star+u_2 u_2^\star  \\
& = -(u_1 u_2^\star+u_2 u_1^\star)+(p(0)+q(0)).1_{\mathcal{A}}.
\end{align*}
Hence, it will suffice to check that the element $x:=u_1 u_2^\star+u_2 u_1^\star$ commutes with both $u_1$ and $u_2$.
First, note that
$$x=(\tr q) u_1+(\tr p) u_2-u_1u_2-u_2u_1=u_1^\star u_2+u_2^\star u_1,$$
and then
compute
$$u_1 x=u_1(u_1^\star u_2+u_2^\star u_1)=p(0)\,u_2+u_1 u_2^\star u_1=x u_1.$$
Symmetrically, one finds that $u_2$ commutes with $x$, which yields the desired result.
\end{proof}

Finally, we set
$$\Lambda_{p,q}:=t^2+\bigl(2\bigl(p(0)+q(0)\bigr)-(\tr p)(\tr q)\bigr) t+F_{p,q}(0),$$
which has its coefficients in $\F$.
Noting that
$$\bigl(t-(x_1-y_1)\id_V\bigr)\bigl(t-(x_2-y_2)\id_V\bigr)=t^2-\delta t+(x_1-y_1)(x_2-y_2)\,\id_V$$
and likewise
$$\bigl(t-(x_1-y_2)\id_V\bigr)\bigl(t-(x_2-y_1)\id_V\bigr)=t^2-\delta t+(x_1-y_2)(x_2-y_1)\,\id_{V},$$
we gather that
$$F_{p,q}(t)=\Lambda_{p,q}(t^2-\delta t).$$
Therefore, if an element $x$ of an $\F$-algebra $\mathcal{A}$ splits into $x=u_1-u_2$ where $p(u_1)=q(u_2)=0$,
then $u_1$ and $u_2$ commute with $F_{p,q}(x)$ by Lemma \ref{lemma:commute}.

As a consequence, we will obtain the following result.

\begin{prop}
Let $(b,u)$ be a symplectic pair.
Then $(b,u)$ is a symplectic $(p,q)$-difference if and only if
both its $(p,q)$-regular and $(p,q)$-exceptional parts are symplectic $(p,q)$-differences.
\end{prop}

\begin{proof}
Denoting by $(b,u)_r$ the $(p,q)$-regular part of $(b,u)$ and by $(b,u)_e$ its $(p,q)$-exceptional part,
we know that $(b,u) \simeq (b,u)_r \bot (b,u)_e$, which yields the converse implication.

For the direct implication, assume that $u=u_1-u_2$ for some $b$-alternating endomorphisms $u_1$ and $u_2$.
As we have just seen, $u_1$ and $u_2$ commute with $F_{p,q}(u)$ and hence they stabilize the linear subspaces $E_{p,q}(u)$ and $R_{p,q}(u)$.
It follows in particular that the induced pairs $(b,u)^{E_{p,q}(u)}$ and $(b,u)^{R_{p,q}(u)}$ are symplectic $(p,q)$-differences.
\end{proof}

It follows that each indecomposable symplectic $(p,q)$-difference is either
regular or exceptional.

\subsection{Main results}\label{section:results}

Before we can state the complete classification of symplectic $(p,q)$-differences, we need
a notation for companion matrices.

\begin{Not}
Given a monic polynomial $r(t)=t^n-\underset{k=0}{\overset{n-1}{\sum}} a_k t^k$ of $\F[t]$,
we denote by
$$C(r):=\begin{bmatrix}
0 &   & & (0) & a_0 \\
1 & 0 & &   & a_1 \\
0 & \ddots & \ddots & & \vdots \\
\vdots & \ddots & \ddots & 0 & a_{n-2} \\
(0) & \cdots & 0 &  1 & a_{n-1}
\end{bmatrix}\in \Mat_n(\F)$$
its \textbf{companion matrix}.
\end{Not}

We are ready to state the classification of indecomposable symplectic $(p,q)$-differences.
As we have just seen, every such pair is either regular or indecomposable.
Moreover, we can use Theorem \ref{theo:Scharlau} to deal with pairs of the form $S(v)$ only.

In each one of the following tables, we give a set of matrices. Each matrix represents an endomorphism
$v$ such that $S(v)$ is an indecomposable symplectic $(p,q)$-difference,
and every indecomposable symplectic $(p,q)$-difference is isometric to $S(v)$ for some $v$ that is represented by one of the matrices.

We start with the regular case, for which the characterization is very simple and requires no discussion (Table \ref{table:1}).

\begin{table}[H]
\caption{The classification of indecomposable regular symplectic $(p,q)$-differences.}
\label{table:1}
\begin{center}
\begin{tabular}{| c | c |}
\hline
Representing matrix & Associated data  \\
\hline
\hline
 & $n\in \N^*$,   \\
$C\bigl(r^n(t^2-\delta t)\bigr)$ & $r \in \F[t]$ irreducible and monic, \\
 & $r(t^2-\delta t)$ has no root in $\Root(p)-\Root(q)$ \\
\hline
\end{tabular}
\end{center}
\end{table}

Next, we tackle the indecomposable exceptional symplectic $(p,q)$-differences.
We start with the case where both $p$ and $q$ are split over $\F$.
The three situations are described in Tables \ref{table:2} to \ref{table:4}.

\begin{table}[H]
\begin{center}
\caption{The classification of indecomposable exceptional symplectic $(p,q)$-differences: When both $p$ and $q$ are split with a double root.}
\label{table:2}
\begin{tabular}{| c | c |}
\hline
Representing matrix & Associated data  \\
\hline
\hline
$C\bigl((t-x)^n\bigr)$ & $n \in \N^*$, $x\in \Root(p)-\Root(q)$ \\
\hline
\end{tabular}
\end{center}
\end{table}

\begin{table}[H]
\begin{center}
\caption{The classification of indecomposable exceptional symplectic $(p,q)$-differences: When both $p$ and $q$ are split with simple roots.}
\label{table:3}
\begin{tabular}{| c | c |}
\hline
Representing matrix & Associated data  \\
\hline
\hline
 & $n \in \N^*$, \\
$C\bigl((t-x)^n\bigr) \oplus C\bigl((t-\delta+x)^n\bigr)$ & $x \in \Root(p)-\Root(q)$  \\
& such that $x \neq \delta -x$ \\
\hline
 & $n \in \N$, \\
$C\bigl((t-x)^{n+1}\bigr) \oplus C\bigl((t-\delta+x)^{n}\bigr)$ & $x \in \Root(p)-\Root(q)$  \\
& such that $x \neq \delta -x$ \\
\hline
 & $n\in \N^*$,  \\
$C\bigl((t-x)^{n}\bigr)$ & $x \in \Root(p)-\Root(q)$ \\
& such that $x =\delta-x$ \\
\hline
\end{tabular}
\end{center}
\end{table}

\begin{table}[H]
\begin{center}
\caption{The classification of indecomposable exceptional symplectic $(p,q)$-differences: When one of $p$ and $q$ is split with a double root
 and the other one is split with simple roots.}
\label{table:4}
\begin{tabular}{| c | c |}
\hline
Representing matrix & Associated data  \\
\hline
\hline
$C\bigl((t-x)^n\bigr) \oplus C\bigl((t-\delta+x)^n\bigr)$ & $n\in \N^*$, $x \in \Root(p)-\Root(q)$  \\
\hline
$C\bigl((t-x)^{n+1}\bigr) \oplus C\bigl((t-\delta+x)^n\bigr)$ & $n\in \N$, $x \in \Root(p)-\Root(q)$  \\
\hline
$C\bigl((t-x)^{n+2}\bigr) \oplus C\bigl((t-\delta+x)^n\bigr)$ & $n\in \N$, $x \in \Root(p)-\Root(q)$  \\
 \hline
\end{tabular}
 \end{center}
\end{table}

Next, the case where $p$ is irreducible but $q$ is split (the case where $p$ is split and $q$ is irreducible is deduced from
it by noting that symplectic $(p,q)$-differences are also symplectic $(q(-t),p(-t))$-differences).
There are two subcases to consider, whether the two polynomials obtained by translating
$p$ along the roots of $q$ are equal or not (Tables \ref{table:5} and \ref{table:6}, respectively).

\begin{table}[H]
\begin{center}
\caption{The classification of indecomposable exceptional symplectic $(p,q)$-differences: When $p$ is irreducible, $q=(t-y_1)(t-y_2)$
for some $y_1,y_2$ in $\F$, and $p(t+y_1)=p(t+y_2)$.}
\label{table:5}
\begin{tabular}{| c | c |}
\hline
Representing matrix & Associated data  \\
\hline
\hline
$C\bigl(p(t+y)^n\bigr)$ &  $n\in \N^*$, $y \in \Root(q)$  \\
\hline
\end{tabular}
 \end{center}
\end{table}

\begin{table}[H]
\begin{center}
\caption{The classification of indecomposable exceptional symplectic $(p,q)$-differences: When $p$ is irreducible, $q=(t-y_1)(t-y_2)$
for some $y_1,y_2$ in $\F$, and $p(t+y_1)\neq p(t+y_2)$.}
\label{table:6}
\begin{tabular}{| c | c |}
\hline
Representing matrix & Associated data  \\
\hline
\hline
$C\bigl(p(t+y_1)^n\bigr) \oplus C\bigl(p(t+y_2)^n\bigr)$ & $n\in \N^*$ \\
\hline
$C\bigl(p(t+y_1)^{n+1}\bigr) \oplus C\bigl(p(t+y_2)^{n}\bigr)$ & $n\in \N$ \\
\hline
$C\bigl(p(t+y_2)^{n+1}\bigr) \oplus C\bigl(p(t+y_1)^{n}\bigr)$ & $n\in \N$ \\
\hline
\end{tabular}
 \end{center}
\end{table}

Next, we consider the situation where both $p$ and $q$ are irreducible in $\F[t]$, with the same
splitting field in $\overline{\F}$ (Table \ref{table:7}).

\begin{table}[H]
\begin{center}
\caption{The classification of indecomposable exceptional symplectic $(p,q)$-differences: When $p$ and $q$ are irreducible with the same splitting field $\L$.}
\label{table:7}
\begin{tabular}{| c | c |}
\hline
Representing matrix & Associated data  \\
\hline
\hline
&  $n\in \N^*$,  \\
$C\Bigl(\bigl(t^2-\delta t+N_{\L/\F}(x-y)\bigr)^{n}\Bigr)$ &  $x \in \Root(p)$, $y \in \Root(q)$ \\
 & with $x-y \not\in \F$  \\
\hline
$C\bigl((t-(x-y))^n\bigr) \oplus C\bigl((t-(x-y))^n\bigr)$ & $n \in \N^*$ odd, $x \in \Root(p)$, \\
& $y \in \Root(q)$ with $x-y \in \F$ \\
\hline
\hline
$C\bigl((t-(x-y))^n\bigr)$ & $n \in \N^*$ even, $x \in \Root(p)$, \\
& $y \in \Root(q)$ with $x-y \in \F$ \\
\hline
\end{tabular}
 \end{center}
\end{table}

We finish with the situation where both $p$ and $q$ are irreducible, with distinct splitting fields in $\overline{\F}$.
There are three subcases to consider.

The first one (which always holds if $\charac(\F) \neq 2$) is the one where
$p$ and $q$ do not have the same discriminant (see Table \ref{table:8}).

\begin{table}[H]
\begin{center}
\caption{The classification of indecomposable exceptional symplectic $(p,q)$-differences: When $p$ and $q$ are irreducible with distinct splitting fields and distinct discriminants.}
\label{table:8}
\begin{tabular}{| c | c |}
\hline
Representing matrix & Associated data  \\
\hline
\hline
$C(F_{p,q}^n)$ & $n \in \N^*$ \\
\hline
\end{tabular}
 \end{center}
\end{table}

The second one occurs when both $p$ and $q$ have $0$ as their discriminant (i.e.\ they are both inseparable). See Table \ref{table:9}.

\begin{table}[H]
\begin{center}
\caption{The classification of indecomposable exceptional symplectic $(p,q)$-differences: When $p$ and $q$ are irreducible with distinct splitting fields, both with a double root.}
\label{table:9}
\begin{tabular}{| c | c |}
\hline
Representing matrix & Associated data  \\
\hline
\hline
$C\Bigl(\bigl(t^2-p(0)-q(0)\bigr)^n\Bigr)$ & $n \in \N^*$ \\
\hline
\end{tabular}
 \end{center}
\end{table}

The last case occurs when $p$ and $q$ have the same discriminant, which is non-zero
(still assuming that $p$ and $q$ have distinct splitting fields in $\overline{\F}$: the combination of all those conditions can occur only if $\charac(\F) = 2$). See Table \ref{table:10}.

\begin{table}[H]
\caption{The classification of indecomposable exceptional symplectic $(p,q)$-differences: When $p$ and $q$ are irreducible with distinct splitting fields and the same discriminant, which is nonzero.}
\begin{center}
\label{table:10}
\begin{tabular}{| c | c |}
\hline
Representing matrix & Associated data  \\
\hline
\hline
$C\Bigl(\bigl(t^2-(\tr p)\,t +p(0)+q(0)\bigr)^n\Bigr)$ & \\
$\oplus$ & $n \in \N^*$ odd \\
$C\Bigl(\bigl(t^2-(\tr p)\,t +p(0)+q(0)\bigr)^n\Bigr)$ &  \\
\hline
$C\Bigl(\bigl(t^2-(\tr p)\,t +p(0)+q(0)\bigr)^n\Bigr)$ & $n \in \N^*$ even \\
\hline
\end{tabular}
 \end{center}
\end{table}

\subsection{Strategy, and structure of the article}

We will use two main techniques to construct symplectic $(p,q)$-differences.
The first one consists in using symplectic extensions:
one takes a $(p,q)$-difference $v$ in an algebra of endomorphisms, and then one immediately obtains that
$S(v)$ is a symplectic $(p,q)$-difference.

The second construction echoes a result from \cite{dSPregular} called the ``Duplication Lemma" (lemma 3.8 there). Given a monic $r \in \F[t]$,
we will construct a symplectic $(p,q)$-difference $(b,u)$ such that $u$ has exactly two invariant factors, both equal to $r(t^2-\delta t)$.

It will turn out (but this requires a case-by-case verification that we will \emph{not} perform) that the following characterization of symplectic $(p,q)$-differences holds true:

\begin{theo}\label{theo:synthesis}
Let $v$ be an endomorphism of a vector space.
The following properties are equivalent:
\begin{enumerate}[(i)]
\item $S(v)$ is a symplectic $(p,q)$-difference.
\item There exist endomorphisms $v_1$ and $v_2$ such that $v \simeq v_1 \oplus v_2$, all the invariant factors of $v_1$ are polynomials in $t^2-\delta t$,
and the endomorphism $v_2$ is a $(p,q)$-difference.
\end{enumerate}
\end{theo}

It turns out however that proving this result requires close to a complete understanding of the
classification of $(p,q)$-differences in algebras of endomorphisms, featured in
\cite{dSPregular} and proved in \cite{dSPregular,dSPsumexceptional}. So far, we have refrained from stating this classification, and for two reasons: firstly, there are lots of cases to consider for exceptional $(p,q)$-differences; secondly, a good deal of this classification turns out to be of no use
for symplectic $(p,q)$-differences, in particular the classification of regular $(p,q)$-differences (which is far more subtle for endomorphisms than for symplectic pairs), as well as the classification of exceptional $(p,q)$-differences when both $p$ and $q$ are irreducible.
Thus, we will only state the results from \cite{dSPregular} and \cite{dSPsumexceptional} when they are useful.

The remainder of the article is laid out as follows.
In Section \ref{section:duplicationlemma}, we prove the Symplectic Duplication Lemma. As a consequence
of it and of the fact that every regular $(p,q)$-endomorphism has all its invariant factors being polynomials in $t^2-\delta t$
(see section 3 of \cite{dSPregular}),  in Section \ref{section:regular} we obtain a very simple characterization of symplectic pairs that are regular $(p,q)$-differences.

The next three sections are devoted to the classification of exceptional symplectic $(p,q)$-differences.
In Section \ref{section:porqsplit}, we obtain such a classification when one of $p$ and $q$ split over $\F$, as a mostly straightforward consequence
of the classification of $(p,q)$-differences in algebra of endomorphisms (the classification is recalled there).

The case where $p$ and $q$ are irreducible, which is more subtle, is dealt with in two separate sections. In Section \ref{section:pandqirrsamesplittingfield}, we deal with the case where $p$ and $q$ are irreducible with the same splitting field:
here, the main difficulty appears when $q$ is a translation of $p$, and it requires an approach that is quite specific to symplectic pairs.
The last case, where $p$ and $q$ are irreducible with distinct splitting fields in $\overline{\F}$, is dealt with in Section \ref{section:pandqirrdiffsplittingfield},
and the results of the former case are used to deal with a very specific situation for fields with characteristic $2$.

As a final remark, we note that in Section \ref{section:results} we have framed our results in terms of indecomposable symplectic pairs.
However, in the remainder of the article we will also give direct characterizations of symplectic $(p,q)$-differences
in terms of invariant factors (where we still differentiate
between regular and exceptional pairs): see Theorems \ref{theo:regular}, \ref{theo:pandqsplitdouble}, \ref{theo:pandqsplitsimple}, \ref{theo:pandqsplitpsimpleqdouble}, \ref{theo:pirrqsplitdouble}, \ref{theo:pirrqsplitsimple},
\ref{theo:pandqirrsamesplittingfield}, \ref{theo:pandqirrdistinctfieldsnonspecial} and \ref{theo:pandqirrdistinctfieldsspecial}.

\section{Constructing symplectic $(p,q)$-differences: the duplication lemma}\label{section:duplicationlemma}

\subsection{Statement of the Symplectic Duplication Lemma}

In \cite{dSPregular} and \cite{dSPsumexceptional}, one of the main tools
for constructing $(p,q)$-differences was the following lemma:

\begin{lemma}[Duplication Lemma, lemma 3.8 in \cite{dSPregular}]\label{lemma:simpleduplicationlemma}
Let $r$ be a nonconstant monic polynomial of $\F[t]$.
Then there exists an endomorphism $u$ of a vector space $V$ such that:
\begin{enumerate}[(i)]
\item $u$ is a $(p,q)$-difference in $\End(V)$.
\item The endomorphism $u$ has exactly two invariant factors: $r(t^2-\delta t)$ and $r(t^2-\delta t)$.
\end{enumerate}
\end{lemma}

In the present work, we will prove a refined version of this result, where we show that, in addition to the endomorphism $u$,
one can find a symplectic form $b$ on $V$ such that $(b,u)$ is a symplectic $(p,q)$-difference:

\begin{lemma}[Symplectic Duplication Lemma]\label{lemma:symplecticduplicationlemma}
Let $r$ be a nonconstant monic polynomial of $\F[t]$.
Then there exists a symplectic pair $(b,u)$ such that:
\begin{enumerate}[(i)]
\item $(b,u)$ is a symplectic $(p,q)$-difference;
\item The endomorphism $u$ has exactly two invariant factors: $r(t^2-\delta t)$ and $r(t^2-\delta t)$.
\end{enumerate}
\end{lemma}

\begin{cor}\label{cor:duplicationlemma}
Let $r$ be a nonconstant monic polynomial of $\F[t]$, and
$v$ be a cyclic endomorphism with minimal polynomial $r(t^2-\delta t)$.
Then $S(v)$ is a symplectic $(p,q)$-difference.
\end{cor}

\begin{proof}[Proof of Corollary \ref{cor:duplicationlemma}, assuming Lemma \ref{lemma:symplecticduplicationlemma}]
We can find a symplectic $(p,q)$-difference $(b,u)$ such that the invariant factors of $u$
are $r(t^2-\delta t),r(t^2-\delta t)$. Yet, the invariant factors of $h(v)$ are also
$r(t^2-\delta t),r(t^2-\delta t)$. Hence $(b,u) \simeq S(v)$ by Theorem \ref{theo:Scharlau},
and we deduce that $S(v)$ is a symplectic $(p,q)$-difference.
\end{proof}

\subsection{Proof of the Symplectic Duplication Lemma}

The proof, like the one of the Duplication Lemma from \cite{dSPregular}, uses the construction of the so-called ``W-algebra'' introduced in \cite{dSPregular} to
classify the regular $(p,q)$-differences in algebras of endomorphisms. Fortunately,
we will not need the most profound results on the W-algebra: the construction of this algebra turns out to be fully sufficient to us.

The proof will also use the following lemma from \cite{dSPregular} (lemma 3.6 there). For a proof, we refer to lemma 14 of
\cite{dSPsumoftwotriang}, where the assumption that $\alpha$ and $\beta$ are nonzero is unnecessary.

\begin{lemma}\label{lemma:blockcompanion}
Let $r$ be a monic polynomial in $\F[t]$ of degree $d>0$.
Then the block matrix $\begin{bmatrix}
0 & C(r)  \\
I_d & \delta\,I_d
\end{bmatrix}$ is cyclic with minimal polynomial $r(t^2-\delta t)$.
\end{lemma}

Now, we recall the basics on the W-algebra.
Let $R$ be a commutative unital $\F$-algebra, and $x$ be an element of $R$.
We write $p(t)=t^2-\lambda t+\alpha$, $q(t)=t^2-\mu t+\beta$, and we
consider the following three matrices of $\Mat_4(R)$:
$$A:=\begin{bmatrix}
0 & -\alpha & 0 & 0 \\
1 & \lambda & 0 & 0 \\
0 & 0 & 0 & -\alpha \\
0 & 0 & 1 & \lambda
\end{bmatrix}, \;
B:=\begin{bmatrix}
0 & -x & -\beta & -\lambda \beta \\
0 & \mu & 0 & \beta \\
1 & \lambda & \mu & \lambda\mu-x \\
0 & -1 & 0 & 0
\end{bmatrix}$$
and
$$C:=AB=\begin{bmatrix}
0 & -\alpha\mu & 0 & -\alpha \beta \\
0 & \lambda\mu-x & -\beta & 0 \\
0 & \alpha & 0 & 0 \\
1 & 0 & \mu & \lambda\mu-x
\end{bmatrix}.$$
One checks that
\begin{equation}\label{basicrel}
AB+BA=\mu A+\lambda B-x I_4, \quad p(A)=0 \quad \text{and} \quad q(B)=0.
\end{equation}
From there, one deduces that $\Vect_R(I_4,A,B,C)$ is stable under multiplication.
Moreover, the observation of the first columns yields that $\Vect_R(I_4,A,B,C)$ is the free $R$-submodule of $\Mat_4(R)$
with basis $(I_4,A,B,C)$.

One defines $\calW(p,q,x)_R$ as the set $\Vect_R(I_4,A,B,C)$ equipped with its structure of $R$-algebra
inherited from the one of $\Mat_4(R)$. We simply write $\calW(p,q,x)$ instead of
$\calW(p,q,x)_R$ when no confusion can arise from the context.
Relations \eqref{basicrel} lead to
$$(A-B)^2-\delta (A-B)=(x-\alpha-\beta)\,I_4.$$

Denote by $d$ the degree of $r$.
By a classical theorem of Frobenius \cite{Frobenius}, there exists an invertible symmetric matrix $s \in \Mat_d(\F)$ such that
$s\,C(r)$ is symmetric.
Now, we work with the commutative $\F$-algebra $R:=\F[C(r)]$, which is isomorphic to the quotient ring $\F[t]/(r)$,
and with the element $x:=\bigl(p(0)+q(0)\bigr)\,1_R+C(r)$. Note that $1_R=I_d$ and that $sx$ is symmetric as a matrix of $\Mat_d(\F)$.

We consider the endomorphisms $a : X \mapsto AX$ and $b:X \mapsto BX$ of the $R$-module $\calW(p,q,x)$.
Since $p(A)=q(B)=0$, we see that $p(a)=q(b)=0$ in $\End_R(\calW(p,q,x))$.
Next, we note that $a$ is represented, in the basis $(I_4,A,B,C)$ of the $R$-module $\calW(p,q,x)$, by $A$
itself. This comes from the observation that $C=AB$ and $A^2=\lambda A-\alpha I_4$
(so that $AC=\lambda AB-\alpha B$).
Next, we also note that $b$ is represented, in the basis $(I_4,A,B,C)$ of the $R$-module $\calW(p,q,x)$, by $B$
itself. For the first and third column, this comes from noting that $B^2=\mu B-\beta I_4$.
For the second column, we write $BA=-AB+\mu A+\lambda B-x I_4$, and for the third column we write
\begin{align*}
B(AB)=(-AB+\mu A+\lambda B-x I_4)\,B & =-AB^2+\mu AB+\lambda B^2- x B \\
& =-A(\mu B-\beta I_4)+\mu AB+\lambda (\mu B-\beta I_4)-x B.
\end{align*}
Next, we set $u=a-b$. We shall prove that the invariant factors of $u$ are $r(t^2-\delta t)$ and $r(t^2-\delta t)$, and
we shall also construct a symplectic form $h$ on the $\F$-vector space $\calW(p,q,x)$ for which
$a$ and $b$ are $h$-alternating.
We start with the second point.

Let us consider the matrix
$$H:=\begin{bmatrix}
0 & 0 & 0 & s  \\
0 & 0 & s & \lambda s \\
0 & -s & 0 & 0 \\
-s & -\lambda s & 0 & 0
\end{bmatrix} \in \Mat_4(\Mat_d(\F)).$$
Since $s$ is symmetric, we see that $H$ is alternating (i.e.\ skewsymmetric with all diagonal entries zero) as a matrix of $\Mat_{4d}(\F)$.
Note that $H$ is also invertible because $s$ is invertible.

Next, we compute that
$$HA=\begin{bmatrix}
0 & 0 & s & \lambda s  \\
0 & 0 & \lambda s & (\lambda^2-\alpha) s \\
-s & -\lambda s & 0 & 0 \\
-\lambda s & (\alpha-\lambda^2) s & 0 & 0
\end{bmatrix}
\quad \text{and} \quad
HB=\begin{bmatrix}
0 & -s & 0 & 0  \\
s & 0 & \mu s & \lambda\mu s-sx \\
0 & -\mu s & 0 & -\beta s \\
0 & sx-\lambda\mu s & \beta s & 0
\end{bmatrix}.$$
Again, since $s$ and $sx$ are symmetric we find that $HA$ and $HB$ are alternating when viewed as matrices
of $\Mat_{4d}(\F)$.
Now, set $y:=C(r)$ in $R$, and consider the basis
$$\mathcal{B}:=(I_4,y I_4,\dots,y^{d-1} I_4,A,y\,A,\dots,y^{d-1}A,B,y\,B,\dots,y^{d-1}\,B,AB,y\,AB,\dots,y^{d-1}\,AB)$$
of the $\F$-vector space $\calW(p,q,x)$. Clearly, the respective matrices of the $\F$-linear maps $a$ and $b$
in that basis are $A$ and $B$, seen as matrices of $\Mat_{4d}(\F)$.
Hence, if we denote by $h$ the symplectic form on the $\F$-vector space $\calW(p,q,x)$
whose matrix in the basis $\mathcal{B}$ equals $H$, the fact that $HA$ and $HB$ are alternating yields that
$a$ and $b$ are $h$-alternating. Hence, $(h,u)$ is a symplectic $(p,q)$-difference.

To complete the proof, we prove that the invariant factors of $a-b$ are $r(t^2-\delta t)$ and $r(t^2-\delta t)$.
This has already been proved in \cite{dSPregular}, but we reproduce the short proof for the reader.
Using $q(B)=0$, it is clearly seen that $\bfC:=(I_4,A-B,B,(A-B)B)$ is still a basis of the free $R$-module $\calW(p,q,x)$.
Using $(A-B)^2=\delta (A-B)+\bigl(x-p(0)-q(0)\bigr) 1_{\Mat_4(R)}$, we get
that the $R$-module endomorphism $u$ is represented in the basis $\bfC$ by the following matrix of $\Mat_4(R)$:
$$\begin{bmatrix}
0 & C(r) & 0 & 0  \\
I_d & \delta\,I_d & 0 & 0 \\
0 & 0 & 0 & C(r) \\
0 & 0 & I_d & \delta\,I_d
\end{bmatrix}.$$
Hence, in some basis of the $\F$-vector space $\calW(p,q,x)$, the endomorphism $u$ is represented by
$$\begin{bmatrix}
0 & C(r)  \\
I_d & \delta\,I_d
\end{bmatrix} \oplus \begin{bmatrix}
0 & C(r)  \\
I_d & \delta\,I_d
\end{bmatrix} \in \Mat_{4d}(\F),$$
and it follows from Lemma \ref{lemma:blockcompanion} that $u$ is the direct sum of two cyclic endomorphisms with minimal polynomial $r(t^2-\delta t)$.
This completes the proof of the Symplectic Duplication Lemma.

\section{The classification of regular symplectic $(p,q)$-differences}\label{section:regular}

Here, we classify the regular symplectic $(p,q)$-differences thanks to the Symplectic Duplication Lemma.

\subsection{Main result}\label{section:regularstatement}

Here, we prove the following result:

\begin{theo}\label{theo:regular}
Let $v$ be a $(p,q)$-regular endomorphism of a vector space.
The following conditions are equivalent:
\begin{enumerate}[(i)]
\item $S(v)$ is a symplectic $(p,q)$-difference.
\item All the invariant factors of $v$ are polynomials in $t^2-\delta t$.
\item All the invariant factors of $S(v)$ are polynomials in $t^2-\delta t$.
\end{enumerate}
\end{theo}

From there, the classification of indecomposable regular symplectic $(p,q)$-differences, as given in Table \ref{table:1}, is easily derived
by using standard techniques on companion matrices (if $r_1,\dots,r_n$ are pairwise relatively prime monic polynomials in $\F[t]$,
then $r_1(t^2-\delta t),\dots,r_n(t^2-\delta t)$ are pairwise relatively prime; when $s_1,\dots,s_n$ are pairwise relatively prime monic polynomials in $\F[t]$, the companion matrix of $\pi:=s_1\cdots s_n$ is similar to $C(s_1) \oplus \cdots \oplus C(s_n)$).

\vskip 3mm
The proof of Theorem \ref{theo:regular} is based on the combination of the Symplectic Duplication Lemma and of the following known result:

\begin{prop}[See \cite{dSPregular} proposition 3.5]\label{prop:invariantregular}
Let $V$ be a vector space and $v$ be a regular $(p,q)$-difference in $\End(V)$.
Then all the invariant factors of $v$ are polynomials in $t^2-\delta t$.
\end{prop}

Note, in this last result, that the converse statement fails
(although it actually holds as soon as one of the polynomials $p$ or $q$ splits over $\F$).
Consider for instance the field $\R$ of real numbers. Choose an arbitrary real number $x$ such that $0<x<4$. Using the classification of
regular $(p,q)$-differences, it was proved in section 3.3 of \cite{dSPregular}
that no cyclic endomorphism with minimal polynomial $t^2+x$ is a $(t^2+1,t^2+1)$-difference
(although such an endomorphism is obviously $(t^2+1,t^2+1)$-regular, and its sole invariant factor is a polynomial in $t^2$).

\subsection{Proof of the main result}

In Theorem \ref{theo:regular}, the implication (i) $\Rightarrow$ (iii) is a straightforward consequence of Proposition \ref{prop:invariantregular}.
Besides, the equivalence between conditions (ii) and (iii) is obvious. Assume finally that
condition (ii) holds. Let us write the invariant factors of $v$ as $r_1(t^2-\delta t),\dots,r_k(t^2-\delta t)$,
where $r_1,\dots,r_k$ are monic polynomials in $\F[t]$.
Then, we can split $v \simeq v_1 \oplus \cdots \oplus v_k$ where
$v_i$ is cyclic with minimal polynomial $r_i(t^2-\delta t)$.
So $S(v) \simeq S(v_1) \bot \cdots \bot S(v_k)$, and by Corollary \ref{cor:duplicationlemma}
each $S(v_i)$ is a symplectic $(p,q)$-difference. We conclude that $S(v)$ is a symplectic $(p,q)$-difference.

\subsection{Disproof of Conjecture \ref{conj:conj2}}

We can now invalidate Conjecture \ref{conj:conj2}.
Consider the field $\R$ of real numbers, and the endomorphism $u : x \mapsto ix$ of the $\R$-vector space $\C$.
Note that $u$ is cyclic with minimal polynomial $t^2+1$ (over the reals).
As seen in Section \ref{section:regularstatement}, $u$ is \emph{not} a $(t^2+1,t^2+1)$-difference in $\End_\R(\C)$.
Yet, by the Symplectic Duplication Lemma $S(u)$ is a symplectic $(t^2+1,t^2+1)$-difference.
This shows that Conjecture \ref{conj:conj2} fails.

However, by using the classification of regular $(p,q)$-differences in algebras of endomorphisms, it can be shown that Conjecture
\ref{conj:conj2} holds in the restricted setting where at least one of the polynomials $p$ and $q$ has a root in $\F$.

\section{Exceptional symplectic $(p,q)$-differences: when one of $p$ and $q$ splits over $\F$}\label{section:porqsplit}

\subsection{When both $p$ and $q$ are split}\label{section:pandqsplit}

\begin{Def}
Let $(u_n)_{n \geq 1}$ and $(v_n)_{n \geq 1}$ be non-increasing sequences of non-negative integers.
Let $p>0$ be a positive integer.
We say that $(u_n)_{n \geq 1}$ and $(v_n)_{n \geq 1}$ are \textbf{$p$-intertwined} when
$$\forall n \geq 1, \; u_{n+p} \leq v_n \quad \text{and} \quad v_{n+p} \leq u_{n.}$$
\end{Def}

\begin{Not}
Given an endomorphism $u$ of a finite-dimensional vector space $V$ over $\F$, a scalar $\lambda \in \F$
and a positive integer $k$, we set
$$n_k(u,\lambda):=\dim \Ker (u-\lambda\, \id_V)^k-\dim \Ker (u-\lambda\, \id_V)^{k-1}$$
i.e.\ $n_k(u,\lambda)$ is the number of cells of size \emph{at least} $k$
associated with the eigenvalue $\lambda$ in the Jordan Canonical Form of $u$.
\end{Not}

Let us recall the classification of exceptional $(p,q)$-differences in algebras of endomorphisms
(see Section 2 of \cite{dSPsumexceptional}).

First of all, if $p$ and $q$ are split then an endomorphism of a vector space is $(p,q)$-exceptional if and only if it is triangularizable
and all its eigenvalues belong to $\Root(p)-\Root(q)$.

\begin{theo}
Assume that both $p$ and $q$ are split with a double root.
Denote by $x$ the sole element of $\Root(p)-\Root(q)$.
Then every $(p,q)$-exceptional endomorphism of a vector space $V$ is a $(p,q)$-difference in $\End(V)$.
\end{theo}

As an immediate corollary, we obtain:

\begin{theo}\label{theo:pandqsplitdouble}
Assume that both $p$ and $q$ are split with a double root.
Let $v$ be a $(p,q)$-exceptional endomorphism of a vector space.
Then $S(v)$ is a symplectic $(p,q)$-difference.
\end{theo}

Next, we consider the case where the two polynomials are split with simple roots.

\begin{theo}[See theorem 2.3 in \cite{dSPsumexceptional}]
Assume that both $p$ and $q$ are split with simple roots.
Let $u$ be a $(p,q)$-exceptional endomorphism of a vector space $V$.
The following conditions are equivalent:
\begin{enumerate}[(i)]
\item $u$ is a $(p,q)$-difference in $\End(V)$.
\item For all $x \in \Root(p)-\Root(q)$ such that $x \neq \delta-x$, the sequences $(n_k(u,x))_{k \geq 1}$
and $(n_k(u,\delta-x))_{k \geq 1}$ are $1$-intertwined.
\end{enumerate}
\end{theo}

From this result, we obtain the following characterization of exceptional symplectic $(p,q)$-differences:

\begin{theo}\label{theo:pandqsplitsimple}
Assume that both $p$ and $q$ are split with simple roots.
Let $v$ be a $(p,q)$-exceptional endomorphism of a vector space $V$.
The following conditions are equivalent:
\begin{enumerate}[(i)]
\item $S(v)$ is a symplectic $(p,q)$-difference;
\item $h(v)$ is a $(p,q)$-difference in $\End(V \times V^\star)$;
\item $v$ is a $(p,q)$-difference in $\End(V)$;
\item For every $x \in \Root(p)-\Root(q)$ such that $x \neq \delta-x$,
the sequences $(n_k(v,x))_{k \geq 1}$ are $(n_k(v,\delta-x))_{k \geq 1}$ are $1$-intertwined.
\end{enumerate}
\end{theo}

\begin{proof}
The implications (i) $\Rightarrow$ (ii) and (iii) $\Rightarrow$ (i) are obvious, and
the equivalence (iv) $\Leftrightarrow$ (iii) is known. So, it suffices to prove that (ii) implies (iv).
Assume that (ii) holds and let $x \in \Root(p)-\Root(q)$ be such that $x \neq \delta-x$.
Then, for all $k \geq 1$, we have $n_k(h(v),x) \geq n_{k+1}(h(v),\delta-x)$ and
$n_k(h(v),\delta-x) \geq n_{k+1}(h(v),x)$. Yet $n_k(h(v),x)=2\, n_k(v,x)$ and
$n_k(h(v),\delta-x)=2\, n_k(v,\delta-x)$ for all $k \geq 1$.
Hence condition (iv) holds.
\end{proof}

With the same method, the following theorem (Theorem \ref{theo:psimpleqdouble1}) is used to obtain the characterization of
the symplectic $(p,q)$-differences when $p$ is split over $\F$ with simple roots, and $q$ is split over $\F$
with a double root (Theorem \ref{theo:pandqsplitpsimpleqdouble}).

\begin{theo}[See theorem 2.2 in \cite{dSPsumexceptional}]\label{theo:psimpleqdouble1}
Assume that $p$ is split with simple roots and that $q$ is split with a double root.
Let $u$ be a $(p,q)$-exceptional endomorphism of a vector space $V$.
The following conditions are equivalent:
\begin{enumerate}[(i)]
\item $u$ is a $(p,q)$-difference in $\End(V)$.
\item With $\Root(p)-\Root(q)=\{x,\delta-x\}$, the sequences $(n_k(u,x))_{k \geq 1}$
and $(n_k(u,\delta-x))_{k \geq 1}$ are $2$-intertwined.
\end{enumerate}
\end{theo}

\begin{theo}\label{theo:pandqsplitpsimpleqdouble}
Assume that $p$ is split with simple roots and that $q$ is split with a double root.
Let $v$ be a $(p,q)$-exceptional endomorphism of a vector space $V$.
The following conditions are equivalent:
\begin{enumerate}[(i)]
\item $S(v)$ is a symplectic $(p,q)$-difference;
\item $h(v)$ is a $(p,q)$-difference in $\End(V \times V^\star)$;
\item $v$ is a $(p,q)$-difference in $\End(V)$;
\item With $\Root(p)-\Root(q)=\{x,\delta-x\}$,
the sequences $(n_k(v,x))_{k \geq 1}$ are $(n_k(v,\delta-x))_{k \geq 1}$ are $2$-intertwined.
\end{enumerate}
\end{theo}

\subsection{When $p$ is irreducible but $q$ splits}\label{section:pirrqsplit}

Assume here that $p$ is irreducible and that $q(t)=(t-y_1)(t-y_2)$ for scalars $y_1,y_2$ in $\F$.
Then $F_{p,q}=p(t+y_1)\,p(t+y_2)$, and $p(t+y_1)$ and $p(t+y_2)$ are irreducible and monic.

If $p(t+y_1)=p(t+y_2)$ (which can happen even if $y_1 \neq y_2$, over fields with characteristic $2$),
then an endomorphism is $(p,q)$-exceptional if and only if its minimal polynomial is a power of $p(t+y_1)$.
And it has been proved in \cite{dSPsumexceptional} (see theorem 3.2 there) that every such endomorphism is a $(p,q)$-difference.
Hence, we obtain the following result:

\begin{theo}\label{theo:pirrqsplitdouble}
Assume that  $p$ is irreducible and that $q(t)=(t-y_1)(t-y_2)$ for scalars $y_1,y_2$ in $\F$
such that $p(t+y_1)=p(t+y_2)$. Then, for every $(p,q)$-exceptional endomorphism $v$ of a vector space,
$S(v)$ is a symplectic $(p,q)$-difference.
\end{theo}

Next, we consider the case where $p(t+y_1) \neq p(t+y_2)$.
In that case, an endomorphism is $(p,q)$-exceptional if and only if its minimal polynomial is the product of a power
of $p(t+y_1)$ with a power of $p(t+y_2)$.

For a monic irreducible polynomial $r \in \F[t]$, denote by $n_k(u,r)$ the number of primary invariants of
the form $r^l$, with $l \geq k$, in the primary canonical form of $u$.
Then, we recall the following result from \cite{dSPsumexceptional}:

\begin{theo}[See theorem 3.3 of \cite{dSPsumexceptional}]
Assume that  $p$ is irreducible and that $q(t)=(t-y_1)(t-y_2)$ for scalars $y_1,y_2$ in $\F$
such that $p(t+y_1)\neq p(t+y_2)$.
Let $u$ be a $(p,q)$-exceptional endomorphism of a vector space $V$.
For $u$ to be a $(p,q)$-difference in $\End(V)$, it is necessary and sufficient that the sequences
$(n_k(u,p(t+y_1)))_{k \geq 1}$ and $(n_k(u,p(t+y_2)))_{k \geq 1}$ be $1$-intertwined.
\end{theo}

From there, by using exactly the same method as in Section \ref{section:pandqsplit}, we conclude:

\begin{theo}\label{theo:pirrqsplitsimple}
Assume that $p$ is irreducible and that $q(t)=(t-y_1)(t-y_2)$ for scalars $y_1,y_2$ in $\F$
such that $p(t+y_1)\neq p(t+y_2)$.
Let $v$ be a $(p,q)$-exceptional endomorphism of a vector space $V$.
The following conditions are equivalent:
\begin{enumerate}[(i)]
\item $S(v)$ is a symplectic $(p,q)$-difference;
\item $h(v)$ is a $(p,q)$-difference in $\End(V \times V^\star)$;
\item $v$ is a $(p,q)$-difference in $\End(V)$;
\item The sequences $(n_k(v,p(t+y_1)))_{k \geq 1}$ are $(n_k(v,p(t+y_2)))_{k \geq 1}$ are $1$-intertwined.
\end{enumerate}
\end{theo}

\subsection{Conclusion}

Forgetting the conditions on the Jordan numbers, we can sum up part of our previous findings as follows:

\begin{theo}
Assume that at least one of $p$ and $q$ has a root in $\F$.
Let $v$ be a $(p,q)$-exceptional endomorphism of a vector space $V$.
Then $v$ is a $(p,q)$-difference in $\End(V)$ if and only if $S(v)$ is a symplectic $(p,q)$-difference.
\end{theo}

From there, we obtain the following corollary:

\begin{theo}
Assume that at least one of $p$ and $q$ has a root in $\F$.
Let $v$ be a $(p,q)$-exceptional endomorphism of a vector space $V$.
Then $v$ is an indecomposable $(p,q)$-difference if and only if $S(v)$ is an indecomposable symplectic $(p,q)$-difference.
\end{theo}

From there, the classification of indecomposable symplectic $(p,q)$-differences stated in Tables \ref{table:2} to \ref{table:6} is
directly obtained from the classification of indecomposable $(p,q)$-differences given in tables 1 to 5 of \cite{dSPsumexceptional}.

\section{Exceptional symplectic $(p,q)$-differences: when $p$ and $q$ are irreducible with the same splitting field}
 \label{section:pandqirrsamesplittingfield}

Here, we assume that both $p$ and $q$ are irreducible, with the same splitting field $\L$ in $\overline{\F}$.

Denote by $\sigma$ the non-identity automorphism of $\L$ over $\F$ if $\L$ is separable over $\F$,
otherwise set $\sigma:=\id_\L$. In any case, splitting $p(t)=(t-x_1)(t-x_2)$ and $q(t)=(t-y_1)(t-y_2)$
in $\L$, we find that $\sigma$ exchanges $x_1$ and $x_2$, and that it exchanges $y_1$ and $y_2$.
Hence, $(x_1-y_1)(x_2-y_2)=(x_1-y_1)\,\sigma(x_1-y_1)=N_{\L/\F}(x_1-y_1)$ and
likewise $(x_1-y_2)(x_2-y_1)=N_{\L/\F}(x_1-y_2)$.
Hence
$$F_{p,q}(t)=\bigl(t^2-\delta t+N_{\L/\F}(x_1-y_1)\bigr)\bigl(t^2-\delta t+N_{\L/\F}(x_1-y_2)\bigr).$$
Next, assume that $F_{p,q}$ has a root in $\F$. Then
we have respective roots $x$ and $y$ of $p$ and $q$, together with a scalar $d\in \F$ such that $x=y+d$.
Hence, $\sigma(x)=\sigma(y)+d$ and it follows that $q(t)=p(t+d)$.
Hence, each monic irreducible divisor of $F_{p,q}$ in $\F[t]$:
\begin{itemize}
\item Is either a polynomial in $t^2-\delta t$;
\item Or equals $t-z$ for some $z \in \Root(p)-\Root(q)$,
in which case $q(t)=p(t+z)$ and $(t-z)^2=t^2-\delta t +(x-y)^2$.
\end{itemize}

Before we tackle the general case, we start with the special case $p=q$, and in this case we
look at the symplectic pairs $(b,u)$ in which $u$ is nilpotent (they are relevant since they are $(p,q)$-exceptional).

Let us recall the following easy lemma from \cite{dSPsumexceptional} (lemma 4.4 there):

\begin{lemma}\label{stablep=qlemma}
Let $u_1$ and $u_2$ be endomorphisms of a vector space such that $p(u_1)=p(u_2)=0$. Then $\Ker(u_1-u_2)$ is stable under $u_1$ and $u_2$.
\end{lemma}

\begin{proof}
Let us write $p(t)=t^2-\lambda t+\alpha$. Let $x \in \Ker(u_1-u_2)$ and set $y:=u_1(x)=u_2(x)$.
Then $u_1(y)=(u_1)^2(x)=\lambda u_1(x)-\alpha x$ and $u_2(y)=(u_2)^2(x)=\lambda u_2(x)-\alpha x$ and hence $u_1(y)=u_2(y)$, i.e.\ $u_1(x) \in \Ker (u_1-u_2)$ and $u_2(x) \in \Ker(u_1-u_2)$.
\end{proof}

\begin{lemma}\label{lemma:nilpotent}
Assume that $p$ is irreducible.
Let $(b,u)$ be a symplectic $(p,p)$-difference in which $u$ is nilpotent.
Then, for every odd integer $k \geq 1$, the number $j_k(u)$ of Jordan cells of $u$ of size $k$ for the eigenvalue $0$
is a multiple of $4$.
\end{lemma}

\begin{proof}
Denote by $V$ the underlying vector space of $(b,u)$.
 By assumption, there exists a $b$-alternating endomorphism $u_1$ of $V$ such that
$p(u_1)=0$. Since $p$ is irreducible, the invariant factors of $u_1$ are all equal to $p$, and
since $u_1$ is $b$-alternating there is an even number $k$ of them. It follows that $\dim V=2k$ is a multiple of $4$.

Next, if $u=0$ then all the Jordan cells of $u$ have size $1$ and $j_1(u)=\dim V$, a multiple of $4$.
We now proceed by induction on the nilindex of $u$.
Assume that $u \neq 0$ and let $u_1,u_2$ be $b$-alternating endomorphisms of $V$ such that $u=u_1-u_2$ and $p(u_1)=p(u_2)=0$.
By Lemma \ref{stablep=qlemma}, $u_1$ and $u_2$ stabilize $\Ker u$. Since they are $b$-alternating, they also stabilize
$(\Ker u)^{\bot_b}=\im u$. Thus, the induced pair $(b,u)^{\im u}$, which we denote by $(\overline{b},\overline{u})$,
is a symplectic $(p,p)$-difference. Moreover, the nilindex of $\overline{u}$ is less than the one of $u$.
Noting that $\overline{u}$ is an endomorphism of $\im u/(\Ker u \cap \im u)$, and working cell by cell, it is not difficult to see
(see e.g.\ section 5.4.1 of \cite{dSPsquarezeroquadratic}) that $\forall k \geq 1, \; j_k(\overline{u})=j_{k+2}(u)$.
Hence, by induction $j_k(u)$ is a multiple of $4$ for every odd $k \geq 3$.

It remains to see that $j_1(u)$ is also a multiple of $4$. For this, we gather from the classification of symplectic pairs that
$j_k(u)$ is even for every $k \geq 1$, and hence $k j_k(u)$ is a multiple of $4$ for all $k \geq 2$.
Finally $\dim V=j_1(u)+\underset{k=2}{\overset{+\infty}{\sum}} k j_k(u)$. Hence $j_1(u)=\dim V-\underset{k=2}{\overset{+\infty}{\sum}} k j_k(u)$
is a multiple of $4$.
\end{proof}

\begin{theo}\label{theo:pandqirrsamesplittingfield}
Assume that $p$ and $q$ are irreducible with the same splitting field in $\overline{\F}$.
Let $(b,u)$ be a $(p,q)$-exceptional symplectic pair. The following conditions are equivalent:
\begin{enumerate}[(i)]
\item $(b,u)$ is a symplectic $(p,q)$-difference;
\item For every $z \in \F \cap (\Root(p)-\Root(q))$ and every odd integer $k \geq 1$,
the number of Jordan cells of $u$ of size $k$ for the eigenvalue $z$ is a multiple of $4$.
\end{enumerate}
\end{theo}

\begin{proof}
Assume first that there exist $b$-alternating endomorphisms $u_1$ and $u_2$ such that
$u=u_1-u_2$, $p(u_1)=0$ and $q(u_2)=0$.
Let $z \in \F \cap (\Root(p)-\Root(q))$. We have seen in the beginning of the present section that
$q(t)=p(t+z)$.
Hence $u-z \id=u_1-(u_2+z \id)$, and $p(u_1)=0=p(u_2+z \id)$.
Replacing $q$ with $p$ and $u$ with $u-z\id$, we are reduced to the situation where $p=q$, and in that situation
we wish to prove that, for every odd $k \geq 1$, the number of Jordan cells of $u$ of size $k$ for the eigenvalue $0$ is a multiple of $4$.

Now, in the present reduced situation, we have $\delta=0$. Hence, the Commutation lemma shows that both $u_1$ and $u_2$ commute with $u^2$.
It follows that $u_1$ and $u_2$ stabilize the nilspace $\Nil(u):=\underset{n \in \N}{\bigcup} \Ker u^n=\underset{n \in \N}{\bigcup} \Ker u^{2n}$.
Since $u$ is $b$-selfadjoint, $\Nil(u)^{\bot}=\underset{n \in \N}{\bigcap} \im u^n$, which is linearly disjoint from $\Nil(u)$, and hence $\Nil(u)$
is $b$-regular.
Hence, in the induced pair $(\overline{b},\overline{u})=(b,u)^{\Nil(u)}$,
the endomorphism $\overline{u}$ has the same Jordan numbers as $u$ for the eigenvalue $0$.
Moreover, since $u_1$ and $u_2$ stabilize $\Nil(u)$, the pair $(\overline{b},\overline{u})$
is a symplectic $(p,p)$-difference. By Lemma \ref{lemma:nilpotent}, for every odd integer $k \geq 1$,
the number of Jordan cells of size $k$ of $\overline{u}$ is a multiple of $4$, and hence so is the one of
$u$ for the eigenvalue $0$.

Conversely, assume that condition (ii) holds.
We have $(b,u)\simeq S(v)$ for some endomorphism $v$ of a vector space. By condition (ii) and the fact that $u$
is $(p,q)$-exceptional, the minimal polynomial of $v$ is a product of irreducible polynomials of the form $r(t^2-\delta t)$ and of polynomials of the form
$t-z$ for $z \in \F \cap (\Root(p)-\Root(q))$ (and then $(t-z)^2=t^2-\delta t+z^2$).
Hence, we can split $v \simeq v_1 \oplus \cdots \oplus v_N$ in which, for all $i \in \lcro 1,N\rcro$:
\begin{itemize}
\item Either $v_i$ is cyclic with minimal polynomial of the form $r^n(t^2-\delta t)$ for some monic $r \in \F[t]$ and some integer $n \geq 1$;
\item Or $v_i$ is cyclic with minimal polynomial $(t-z)^{2k}=(t^2-\delta t+z^2)^k$ for some
$z \in \F \cap (\Root(p)-\Root(q))$ and some integer $k \geq 1$;
\item Or $v_i$ is the direct sum of two cyclic endomorphisms with minimal polynomial
$(t-z)^{2k+1}$ for some
$z \in \F \cap (\Root(p)-\Root(q))$ and some integer $k \geq 0$.
\end{itemize}
In the first two cases, the Symplectic Duplication Lemma shows that $S(v_i)$ is a symplectic $(p,q)$-difference.
In the last one, the classification of $(p,q)$-differences in endomorphisms, given in section 4 of \cite{dSPsumexceptional}, directly shows
that $v_i$ is a $(p,q)$-difference, and hence $S(v_i)$ is a symplectic $(p,q)$-difference.
We conclude that $(b,u) \simeq S(v) \simeq S(v_1) \bot \cdots \bot S(v_N)$ is a symplectic $(p,q)$-difference.
\end{proof}

As a consequence, we have:

\begin{cor}
Assume that $p$ and $q$ are irreducible with the same splitting field in $\overline{\F}$.
Let $v$ be a $(p,q)$-exceptional endomorphism of a vector space. The following conditions are equivalent:
\begin{enumerate}[(i)]
\item $S(v)$ is a symplectic $(p,q)$-difference;
\item For every $z \in \F \cap (\Root(p)-\Root(q))$ and every odd integer $k \geq 1$,
the number of Jordan cells of $v$ of size $k$ for the eigenvalue $z$ is even.
\end{enumerate}
\end{cor}

From this corollary, the classification of indecomposable symplectic $(p,q)$-differences, as given in Table \ref{table:7}, is obvious.

\section{Exceptional symplectic $(p,q)$-differences: when $p$ and $q$ are irreducible with distinct splitting fields}
\label{section:pandqirrdiffsplittingfield}

Here, we assume that $p$ and $q$ are irreducible over $\F$, with distinct splitting fields in $\overline{\F}$.
As seen in the beginning of section 5 of \cite{dSPsumexceptional}, there are two main cases for $F_{p,q}$:

\begin{itemize}
\item If $p$ and $q$ have distinct discriminants (which is always the case if $\charac(\F) \neq 2$)
then $F_{p,q}$ is irreducible over $\F$.

\item If $p$ and $q$ have the same discriminant, then $\charac(\F)=2$ and
$F_{p,q}=(t^2+(\tr p)\, t+p(0)+q(0))^2$, and $t^2+(\tr p)\ t+p(0)+q(0)$ is irreducible.
Then, there are two additional subcases, as here $\delta=0$ and $t^2+(\tr p) t+p(0)+q(0)$ is a polynomial in $t^2-\delta t$
if and only if $\tr p=0$. So, either both $p$ and $q$ are separable and then
$t^2+(\tr p) t+p(0)+q(0)$ is not a polynomial in $t^2-\delta t$, or both $p$ and $q$ are inseparable
and then $t^2+(\tr p) t+p(0)+q(0)=t^2+p(0)+q(0)$ is a polynomial in $t^2-\delta t$.
\end{itemize}

Hence, except when $\charac(\F)=2$ and $\tr p=\tr q \neq 0$, the invariant factors of a $(p,q)$-exceptional endomorphism
$u$ are polynomials in $t^2-\delta t$, and the Symplectic Duplication Lemma allows us to conclude directly:

\begin{theo}\label{theo:pandqirrdistinctfieldsnonspecial}
Assume that $p$ and $q$ are irreducible, with distinct splitting fields in $\overline{\F}$. Assume furthermore that
$\charac(\F) \neq 2$ or $\tr(p)=\tr(q)=0$ or $\tr(p) \neq \tr(q)$.
Then for every $(p,q)$-exceptional endomorphism $v$, the symplectic pair $S(v)$
is a symplectic $(p,q)$-difference.
\end{theo}

It remains to consider the special case where $\charac(\F)=2$ and $\tr p =\tr q \neq 0$.

\begin{theo}\label{theo:pandqirrdistinctfieldsspecial}
Assume that $p$ and $q$ are irreducible with distinct splitting fields in $\overline{\F}$, that
$\charac(\F)=2$ and that $\tr p=\tr q\neq 0$.
Let $(b,u)$ be a symplectic pair in which $u$ is $(p,q)$-exceptional.
The following conditions are equivalent:
\begin{enumerate}[(i)]
\item $(b,u)$ is a symplectic $(p,q)$-difference;
\item For every odd integer $k \geq 1$, the number of invariant factors of $u$
that equal $(t^2-(\tr p)\, t+p(0)+q(0))^k$ is a multiple of $4$.
\end{enumerate}
\end{theo}

\begin{proof}
Remember from theorem 5.2 of \cite{dSPsumexceptional} that, for every integer $k \geq 1$,
every endomorphism that is the direct sum of two cyclic endomorphisms with minimal polynomial
$\bigl(t^2-\tr (p)\, t+p(0)+q(0)\bigr)^k$ is a $(p,q)$-difference. Noting that
$\bigl(t^2-\tr (p)\, t+p(0)+q(0)\bigr)^{2k}=F_{p,q}^k$ is a polynomial in $t^2-\delta t$ for all $k \geq 1$, we can therefore
use the same technique as in the proof of Theorem \ref{theo:pandqirrsamesplittingfield} to obtain that condition (ii) implies condition (i),
with the help of the Symplectic Duplication Lemma.

Conversely, assume that $(b,u)$ is a symplectic $(p,q)$-difference.
Denote by $\L$ the splitting field of $pq$ in $\overline{\F}$. Over $\L$, the polynomial
$r:=t^2-\tr (p) t+p(0)+q(0)$ splits into $r=(t-x)(t-y)$, where $x \neq y$.
Hence $r$ is split over $\K:=\F[x]$ with distinct roots. Finally, $x \in \Root(p)-\Root(q)$.
Now, consider the extended symplectic pair $(b,u)_\K$ (through the natural extension of scalars technique),
which is still a symplectic $(p,q)$-difference.
Note that $u_\K$ has the same invariant factors as $u$. Let $k \geq 1$ be an odd integer.
The number of Jordan cells of $u_\K$ of size $k$ for the eigenvalue $x$ equals the number of invariant factors of $u$
that equal $r^k$. By Theorem \ref{theo:pandqirrsamesplittingfield}, the former is a multiple of $4$, and hence so is the latter.
This completes the proof.
\end{proof}

As a corollary, we obtain:

\begin{theo}
Assume that $p$ and $q$ are irreducible with distinct splitting fields in $\overline{\F}$, that
$\charac(\F)=2$ and that $\tr p=\tr q\neq 0$.
Let $v$ be a $(p,q)$-exceptional endomorphism of a vector space.
The following conditions are equivalent:
\begin{enumerate}[(i)]
\item $S(v)$ is a symplectic $(p,q)$-difference;
\item For every odd integer $k \geq 1$, an even number of invariant factors of $v$ equal $(t^2-\tr (p) t+p(0)+q(0))^k$.
\end{enumerate}
\end{theo}

From there, recovering the classification of the indecomposable exceptional symplectic $(p,q)$-differences, as given in
Table \ref{table:10}, is an easy task.

\end{document}